\documentclass[12pt]{amsart}
\usepackage[utf8]{inputenc}
\usepackage{amsmath}
\usepackage{amssymb}
\usepackage{amsthm}
\usepackage{hyperref}
\usepackage{authblk}
\usepackage{doi}

\usepackage{xcolor}

\newtheorem{theorem}{Theorem}
\newtheorem{lemma}{Lemma}
\newtheorem{prop}{Proposition}

\theoremstyle{definition}
\newtheorem{defn}{Definition}

\begin{document}
\begin{center}
\uppercase{\bf \bf{Minimally Intersective Polynomials with Arbitrarily Long Factorization}}
\vskip 20pt
{\bf Bhawesh Mishra \footnote{The author was partly supported by the NSF, under grant DMS-1812028.}}\\
{ Department of Mathematics, The Ohio State University, Columbus, OH USA}\\
{\tt mishra.188@osu.edu}\\
\end{center}
\vskip 20pt

\centerline{\bf Abstract}
\noindent
Given a natural number $n \geq 4$, we show that there exist infinitely many polynomials $f_{n}(x):= (x^{2} - a_{1}) (x^{2} - a_{2}) \cdots (x^{2} - a_{n})$ such that $(i)$ $f_{n}(x)$ has a root modulo every positive integer, $(ii)$ $f_{n}(x)$ has no rational roots, and $(iii)$ every proper divisor of $f_{n}(x)$ fails to have a root modulo some positive integer. We will call such polynomials \textit{minimally intersective}. Our proof also shows that once $a_{1}, a_{2}, \ldots, a_{n-1}$ are chosen, the set of natural numbers $a_{n}$ such that the polynomial $f_{n}(x):= (x^{2} - a_{1}) (x^{2} - a_{2}) \cdots (x^{2} - a_{n})$ is minimally intersective has positive asymptotic density in $\mathbb{N}$.

\section{Introduction}
A set $S \subset\mathbb{Z}$ is called intersective if given any set T $\subset\mathbb{Z}$ with positive upper density, one has $S$ $\cap$ $(T - T) \not\subseteq \{0\}$. Given a $T \subset\mathbb{Z}$, $(T - T)$ is defined as $\{ t_{1} - t_{2} : t_{1}, t_{2} \in T\}$ and the upper density of $T$ is defined as
\begin{equation*}
\overline{d} (T) := \limsup_{n\rightarrow\infty} \frac{|T \cap \{-n,\ldots , -2,-1, 0, 1, 2, \ldots , n\}|}{2n+1}.
\end{equation*} 
A polynomial $f(x_{1}, \ldots , x_{m}) \in\mathbb{Z} [x_{1}, \ldots , x_{m}]$ is said to be intersective if the set of its values $\{ f(x_{1}, \ldots , x_{m}) : x_{1}, \ldots , x_{m}\in\mathbb{Z} \}$ is intersective. S\'ark\"ozy and Furstenberg independently and concurrently proved, in \cite{Sa1} and \cite{Fu} respectively, that for any subset $T$ of integers with positive upper density, the set $(T - T)$ contains a perfect square. In other words, they proved that the polynomial $f(x) = x^{2}$ is intersective. 

Kamae and Mend\'es-France, in \cite{KaMF}, showed that a polynomial $f$ of one variable is intersective if and only if $f(x) \equiv 0\hspace{1mm} (\text{mod } m)$ is solvable for every positive integer $m > 1$. A very special case of the polynomial Szemer\'edi's theorem, obtained by Bergelson, Leibman and Lesigne in \cite{BLL}, generalizes this fact to the polynomials of many variables. Their result implies that a polynomial $g(x_{1}, \ldots , x_{k}) \in\mathbb{Z}[x_{1}, \ldots , x_{k}]$ is intersective if and only if the congruence $g(n_{1}, n_{2}, \ldots , n_{k}) \equiv 0 \hspace{1mm} (\text{mod } m)$ is solvable for every $m > 1$.

Berend and Bilu, in \cite{BerBil}, obtained a criterion for any polynomial $f$ of one variable to be intersective. An implication of their result is that any polynomial in one variable that is intersective, but has no rational root, has to be of degree greater than 4. On the other hand, there are single-variable intersective polynomials of degree greater than $4$ that have no rational roots. Hyde, Lee and Spearman obtained an infinite family of intersective polynomials of the form
\begin{gather*}
h(x) = (x^{3} - n) (x^{2} + 3), 
\end{gather*}
none of which have rational roots \cite{HLS}. One can easily show that if $p, q$ are distinct odd primes such that $p \equiv q \equiv 1 \hspace{1mm} (\text{mod } 4)$ and $\big(\frac{p}{q}\big) = +1$ then the polynomial
\begin{gather*}
f(x) = (x^{2} - p) (x^{2} - q) (x^{2} - pq)
\end{gather*}
is intersective. Hyde and Spearman generalized this result to the case when $p$ and $q$ are replaced by square-free integers. Let $c$ and $d$ be square-free integers not equal to $1$, let $c_{1} = \frac{c}{\text{gcd}(c,d)}$ and $d_{1} = \frac{d}{\text{gcd}(c,d)}$. Hyde and Spearman obtained necessary and sufficient conditions for the polynomial \begin{gather*}
    p(x) = (x^{2} - c) (x^{2} - d) (x^{2} - c_{1}d_{1})
\end{gather*}
to be intersective but have no rational root \cite{HS}. This result was extended in \cite{Mishra} by obtaining a necessary and sufficient condition for polynomials of the form
\begin{gather*}
    f(x) = (x^{2} - a_{1}) (x^{2} - a_{2}) \cdots (x^{2} - a_{n})
\end{gather*}
to be intersective without having a rational root (see Proposition $1$ below). Here $n \geq 3$ and $a_{1}, a_{2}, \ldots, a_{n}$ are square-free integers not equal to $1$. 

Note that if $f(x)$ is an intersective polynomial then for every polynomial $g(x)$, $f(x) g(x)$ is also intersective. Therefore, one could always get more examples of intersective polynomials by multiplying a given intersective polynomial $f(x)$ by any polynomial $g(x)$. The topic of this article is to show existence of, and construct, a particular type of intersective polynomial, which is defined below.  
\begin{defn}
A polynomial $f(x) \in\mathbb{Z} [x]$ is said to be \textit{minimally intersective} if it satisfies the following two conditions:
\begin{enumerate}
    \item $f(x)$ is intersective but $f(x)$ does not have a rational root.
    
    \item None of the proper divisors of $f(x)$ is intersective.
\end{enumerate}
\end{defn}
Minimally intersective polynomials can be thought of as genuinely new examples of intersective polynomials because they are not obtained by adjoining factors to an already intersective polynomials. The result in this article shows that for every $n \geq 4$, there exist minimally intersective polynomials with $n$ quadratic factors. 

Given an integer $n$ and a prime $p$ we will denote the group of quadratic residues modulo $p$ by $Q_{p}$ and the Legendre symbol of $n$ with respect to $p$ by $\big(\frac{n}{p}\big)$. Similarly, $p^{a} \mid\mid n$ $(a \geq 1)$ will denote that $p^{a}$ is the highest power of prime $p$ dividing the integer $n , (a,b)$ will denote the greatest common divisor of two natural numbers $a$ and $b$, and rad$(x)$ will denote the square-free part of natural number $x$. The asymptotic density of a set $A \subset\mathbb{N}$ is defined as:
\begin{equation*}
    \lim_{n\rightarrow\infty} \frac{|A \cap \{1, 2, \ldots, n\}|}{n}.
\end{equation*} 
Let $p$ be an odd prime, $e \geq 1$ and $a \in\mathbb{Z}$ such that $p \nmid a$. Then $a$ is a square modulo $p^{e}$ if and only if $\big(\frac{a}{p}\big) = +1 $. This fact is an immediate consequence of Hensel's lemma for the polynomial $(x^{2} - a)$. The main result of this article is the following theorem. 
\begin{theorem}
For every $n \geq 4$, natural numbers $\{a_{i}\}_{i=1}^{n}$ can be chosen such that the polynomial $f_{n}(x):= (x^{2} - a_{1}) \cdots (x^{2} - a_{n})$ is minimally intersective. In fact, once $a_{1}, a_{2}, \ldots, a_{n-1}$ are chosen, the set of natural numbers $a_{n}$ for which $f_{n}(x)$ is minimally intersective, has positive asymptotic density. 
\end{theorem}
We will collect some preliminary results in Section $2$. In Section $3$, we shall describe a process to find square-free integers $a_{1}, \ldots , a_{n}$ that define the polynomials $f_{n}(x)$ for the corresponding $n \geq 4$. Section $4$ contains proof that the polynomial $f_{n}(x)$ thus obtained is intersective. The proof that $f_{n}(x)$ is minimally intersective is contained in Section $5$. Section $6$ contains an explicit example of $f_{4}(x)$ and $f_{5}(x)$ each. 
\section{Preliminaries}
We will repeatedly utilize the following characterization of intersectivity of polynomials consisting of quadratic factors that is proved in \cite{Mishra}.
\begin{prop}
Let $n \geq 3$ and let $a_{1}, a_{2}, \ldots , a_{n}$ be distinct nonzero square-free integers, none of which is $1$. Then the polynomial $f(x) = \prod_{i=1}^{n} (x^{2} - a_{i})$ is intersective if and only if the following conditions are satisfied.
\begin{enumerate}
    \item There exists $T \subset \{1, 2, \ldots , n\}$ of odd cardinality such that:
    \begin{enumerate}
    \item the product $\prod_{j \in T} a_{j}$ is a perfect square, and
    
    \item for every $j \in T$ and for every odd prime $p$ dividing $a_{j}$, there exists $i \in \{1, \ldots , n\}$, $i \neq j$, such that $\big(\frac{a_{i}}{p}\big) = +1 $ . 
    \end{enumerate}
    
    \item One of the $a_{i}$ is of the form $8m + 1$ for some $m \in\mathbb{Z}$ and $m \neq 0$.
\end{enumerate}
\end{prop}
We will also use the following classical result about distribution of square-free integers in an arithmetic progressions that was originally proved in \cite{Prachar}. 

\begin{prop}
Let $a, b \in\mathbb{N}$ such that $(a,b)$ is square-free. Then the density of the set of square-free natural numbers congruent to $b$ modulo $a$ is $\frac{6}{\pi^{2}} \prod_{p\mid a} \big(1 - \frac{1}{p^{2}} \big)^{-1}$. 
\end{prop}
For the sake of brevity, we will present the following elementary fact (without proof) as a lemma.
\begin{lemma}
Let $n \geq 3$ and let $a_{1}, a_{2}, \ldots , a_{n}$ be distinct square-free nonzero integers, none of which is equal to $1$. Then $\prod_{i=1}^{n} a_{i}$ is a perfect square if and only if for every $j \in \{1, 2, \ldots , n\}$, $a_{j} = rad$  $\big(\prod_{i=1, i \neq j}^{n} a_{i}\big)$.  
\end{lemma}
Now we will state and prove some elementary number-theoretic lemmas that we shall repeatedly utilize in our proofs. 
\begin{lemma}
Let $p$ be an odd prime and let $a \in\mathbb{Z}$ be a square-free integer. If $\big(\frac{a}{p}\big) \neq +1$ then $a$ cannot be a square modulo $p^{2}$. 
\end{lemma}

\begin{proof}
If $p \nmid a$ then $\big(\frac{a}{p}\big) \neq +1$ implies that $a$ cannot be a square modulo $p^{2}$. On the other hand, if $p \mid a$ then $p \mid\mid a$ because $a$ is square-free. 

Assume for the sake of contradiction that $a$ is a square modulo $p^{2}$ , i.e.,  $x^{2} \equiv a \hspace{1mm} (\text{mod}\hspace{1mm} p^{2})$ for some $x \in\mathbb{Z}$. Then we have $p^{2} \mid (x^{2} - a)$ , i.e.,  $p \mid (x^{2} - a)$. Since $p \mid a$ and $p \mid (x^{2} - a)$ we have that $p \mid x^{2}$ implying $p^{2} \mid x^{2}$. However, $p^{2} \mid (x^{2} - a)$ and $p^{2} \mid x^{2}$ gives that $p^{2} \mid a$, contradicting that $a$ is square-free. Therefore, $a$ cannot be a square modulo $p^{2}$.  
\end{proof}

\begin{lemma}
Let $k, m \in\mathbb{N}$, let $p$ be a prime and let $f(x) =  \prod_{i=1}^{m} (x^{2} - a_{i}) \in\mathbb{Z}[x]$. If $(x^{2} - a_{i}) \equiv 0 \hspace{1mm} (\text{mod}\hspace{1mm} p^{k})$ is not solvable for any $1 \leq i \leq m$ then $f(x) \equiv 0 \hspace{1mm} (\text{mod}\hspace{1mm} p^{km})$ is not solvable.  
\end{lemma}

\begin{proof}
For the sake of contradiction, assume that $f(x) \equiv 0 \hspace{1mm} (\text{mod}\hspace{1mm} p^{km})$ is solvable for some $x \in\mathbb{Z}$ , i.e.,  $p^{km} \mid (x^{2} - a_{1}) \cdots (x^{2} - a_{m})$. Then we must have $p^{k} \mid (x^{2} - a_{j})$, for some $j \in \{1, 2, \ldots , m\}$. 

However $p^{k} \mid (x^{2} - a_{j})$ implies that $(x^{2} - a_{j}) \equiv 0 \hspace{1mm} (\text{mod}\hspace{1mm} p^{k})$ is solvable, a contradiction to the fact that $(x^{2} - a_{i}) \equiv 0 \hspace{1mm} (\text{mod}\hspace{1mm} p^{k})$ is not solvable for any $i$. Therefore, we have the result. 
\end{proof}

\begin{lemma}
Let $k \geq 3$ and let $a_{1}, \ldots , a_{k}$ be distinct, nonzero square-free integers. Suppose that for each $ 1 \leq m \leq k$ and for every $T \subset \{1, 2, \ldots , m\}$, $a_{m} > \text{rad } \big( \prod_{j \in T, j \neq m} a_{j} \big)$. Then for any subset $S$ $\subseteq \{1, 2, \ldots , k\}$, $\prod_{j \in S} a_{j}$ is not a perfect square. 
\end{lemma}

\begin{proof}
For the sake of contradiction, assume that there is  an $S \subseteq \{1, 2, \ldots , k\}$ such that $\prod_{j \in S} a_{j}$ is a perfect square. Then using Lemma $1$, we have that 
\begin{equation*}
    a_{j_{0}} = \text{rad } \big(\prod_{j \in S, j \neq j_{0}} a_{j}\big),
\end{equation*}
where $j_{0} = $ max $S$. This is a contradiction to our assumption for $m = j_{0}$ and $T = S \subset \{1, 2, \ldots , j_{0} \}$. Hence we have the result. 
\end{proof}

\section{Finding $a_{1}, a_{2}, \ldots , a_{n}$ such that $f_{n}(x) = \prod_{i=1}^{n} (x^{2} - a_{i})$}

In this section, we will show the existence of natural numbers $a_{1}, a_{2}, \ldots , a_{n}$ that will define the polynomial $f_{n}(x) = (x^{2} - a_{1}) (x^{2} - a_{2}) \cdots (x^{2} - a_{n})$. Each of the integers $a_{1}, \ldots , a_{n}$ will be square-free and not equal to $1$. They can be chosen in accordance with the following steps. 

\begin{enumerate}

    \item Pick distinct odd primes $p_{1}$ and $p_{2}$. For each $i = 1, 2$ pick nonzero elements $b_{i},c_{i} \in \big( \mathbb{Z}/p_{i}\mathbb{Z}\big)^{*}$ such that $c_{i} \in Q_{p_{i}}$ $b_{i} \not\in Q_{p_{i}}$. Set $a_{1} = p_{1}p_{2}$.
    
    We can choose $a_{1}$ as any square-free odd natural number that has at least two odd prime factors and choose $p_{1}, p_{2}$ as two odd primes dividing $a_{1}$. For the sake of brevity, we choose $a_{1}$ to be a product of two odd primes.
    
    \item Pick square-free $a_{2} \in\mathbb{N}$ such that $ a_{2} > a_{1}$ and $a_{2} \equiv b_{j}  \hspace{1mm} (\text{mod} \hspace{1mm} p_{j})$ for $j = 1,2$. Any integer $a_{2}$ satisfying $a_{2} \equiv b_{j}  \hspace{1mm} (\text{mod} \hspace{1mm} p_{j})$ for $j = 1,2$ above is unique modulo $p_{1}p_{2}$, as a consequence of the Chinese remainder theorem. 
    
    Since $p_{1} \neq p_{2}$, $p_{1}p_{2}$ is square-free and hence infinitely many square-free $a_{2}$ exist, by Proposition $2$. We pick a square-free $a_{2}$ greater than $a_{1}$. Also note that since $(b_{j}, p_{j}) = 1$ for every $j = 1, 2$ we have that $(a_{2}, p_{j}) = 1$ for $j = 1, 2$.
    
    \item For any $i \leq n-3$, after choosing $a_{1}, \ldots , a_{i-1}$, choose a square-free integer $ a_{i}$ such that:
    
    \begin{itemize}
    
        \item $a_{i} \equiv b_{j}  \hspace{1mm} (\text{mod} \hspace{1mm} p_{j})$ for $j = 1, 2$ and 
        
        \item for every $A \subset \{1, \ldots , i-1\}$, $a_{i}$ is greater than the rad $\big(\prod_{j \in A} a_{j}\big)$.
        
    \end{itemize}

    Exactly as in step $2$, a square free $a_{i}$ satisfying the above exists as a consequence of Proposition $2$ and $(a_{i}, p_{j}) = 1$ for $j = 1, 2$.
    
    \item Choose a square-free natural number $a_{n-2}$ that satisfies the following requirements.
    
    \begin{itemize}
    \item \begin{equation*}
        a_{n-2} \equiv \begin{cases}
        
        b_{1}  \hspace{1mm} (\text{mod} \hspace{1mm} p_{1}) ; & \text{if $n$ is even}\\
        
        c_{1}  \hspace{1mm} (\text{mod} \hspace{1mm} p_{1}) ; & \text{if $n$ is odd}
        
        \end{cases}
    \end{equation*}
    
    \item $a_{n-2} \equiv b_{2} \hspace{1mm} (\text{mod}\hspace{1mm} p_{2})$
    
    \item $a_{n-2}$ is greater than rad $\big(\prod_{j \in T} a_{j}\big)$, for any $T \subset \{1, \ldots , n-3 \}$. 
    
    \end{itemize}
    
    Such a square-free $a_{n-2}$ exists, again due to Proposition $2$. Similarly, we also have $(a_{n-2}, p_{j}) = 1$ for $j = 1, 2$. 
    
    \item Define square-free natural number $a_{n-1}$ as:
    
    \begin{equation*}
        a_{n-1} = \begin{cases}
        
        \text{rad } \big(\prod_{i=1}^{n-2} a_{j}\big); & \text{ if $n$ is even}\\
        
        \text{rad } \big(\prod_{i=1}^{n-3} a_{j}\big); & \text{ if $n$ is odd.}
        
        \end{cases}
      \end{equation*}
      
      \item Choose all the odd primes $p_{1}, \ldots , p_{M}$ dividing any of $a_{1}, \ldots , a_{n-1}$ and pick $c_{j} \in Q_{p_{j}}$ for every $1 \leq j \leq M$. Now, pick a square-free natural number $a_{n}$ that satisfies the following.
      
      \begin{itemize}
      
          \item $a_{n} \equiv c_{j}  \hspace{1mm} (\text{mod} \hspace{1mm} p_{j})$ for any $2 \leq j \leq M$
          
          \item \begin{equation*}
              a_{n} \equiv \begin{cases}
              
              c_{1}  \hspace{1mm} (\text{mod} \hspace{1mm} p_{1}); & \text{ if $n$ is even}\\
              
              b_{1}  \hspace{1mm} (\text{mod} \hspace{1mm} p_{1}); & \text{ if $n$ is odd}
              
              \end{cases}
          \end{equation*}
          
          \item $a_{n} \equiv 1 \hspace{1mm} (\text{mod}\hspace{1mm} 8)$ 
          
          \item For any subset $S \subset \{1, 2, \ldots , n-1\}$, $a_{n} > \text{ rad }\big(\prod_{j \in S} a_{j}\big)$. 
          
      \end{itemize}
      
      Any integer $a_{n_{0}}$ that satisfies above conditions is unique modulo $(8p_{1} \cdots p_{M})$, as a consequence of the Chinese remainder theorem. In other words, if $a_{n_{0}}$ satisfies above congruences then any integer in the arithmetic progression $\big(8p_{1}\cdots p_{M}\big) \mathbb{N} + a_{n_{0}} \big)$ also satisfies those congruences.  
      
      Since $a_{n_{0}}$ is odd and $(c_{i}, p_{i}) = 1$ for $i = 1, 2, \ldots, M$, $\big(a_{n_{0}}, 8p_{1}\cdots p_{M}\big)$ is square-free. Hence the set of square-free natural numbers $a_{n}$ that satisfy above congruences is of positive density, as a consequence of Proposition $2$. We choose a square-free $a_{n}$ such that for any $S \subset \{ 1, 2, \ldots , n-1\}$:
      
      \begin{gather*}
          a_{n} > \text{rad }\big(\prod_{j \in S} a_{j}\big)
      \end{gather*}

      \item Define $f_{n}(x) := (x^{2} - a_{1}) (x^{2} - a_{2}) \cdots (x^{2} - a_{n})$. 
\end{enumerate}

\section{Proof that $f_{n}(x)$ is Intersective}
In this section, we shall prove that $f_{n}(x)$ has roots modulo every integer, by showing that $f_{n}(x)$ satisfies the conditions in Proposition $1$.

\begin{enumerate}
    \item Condition $1(a)$ of Proposition $1$ holds for $f_{n}(x)$ because $\prod_{j\in T} a_{j}$ is a perfect square, where

\begin{equation*}
    T = \begin{cases}
    
    \{1, 2, \ldots , n-2, n-1\} & \text{ ; if $n$ is even}\\
    
    \{1, 2, \ldots , n-4, n-3, n-1 \} & \text{ ; if $n$ is odd.}
    
    \end{cases}
\end{equation*}

This follows from Step $5$ of Section $3$ and Lemma $1$. 

   \item Let $j \in T $ and $p$ be any odd prime dividing $a_{j}$, then from the step $6$ of Section $3$ it follows that $p = p_{j}$ for some $1 \leq j \leq M$. It also follows from steps $4$ and $6$ of Section $3$ that: 
   
   \begin{equation*}
      \begin{cases}
         \big(\frac{a_{n}}{p}\big) = +1 & \text{if $j \neq 1$ or $n$ is even}\\
         
         \big(\frac{a_{n-2}}{p}\big) = +1 & \text{if $j = 1$ and $n$ is odd.}
    
      \end{cases}
   \end{equation*}
   
   Therefore the condition $1(b)$ of Proposition $1$ is also satisfied for $f_{n}(x)$. 
   
   \item The condition $2$ of the Proposition $1$ is satisfied for $f_{n}(x)$ because $a_{n}$ is chosen to be square-free and equivalent to $1$ modulo $8$ (in Step $6$ of Section $3$). 
\end{enumerate}

\section{Proof that $f_{n}(x)$ is Minimally Intersective}
In this section, we shall prove that if we remove any quadratic factors from $f_{n}(x)$, the resulting polynomial $g_{n}(x)$ will fail to be intersective. To show that $g_{n}(x)$ is not intersective, we will show that $g_{n}(x)$ fails to satisfy the necessary conditions in Proposition $1$. We will separate the proofs into parts according to the quadratic factor that is being removed from $f_{n}(x)$.

\begin{itemize}
    \item\underline{\textbf{Removing $(x^{2} - a_{r})$ for any $1 \leq r \leq (n-3)$}}
    
    In Section $3$, we chose the square-free integers $a_{1}, a_{2} \ldots , a_{r-1}, a_{r+1}, \ldots , a_{n-2}, a_{n}$ such that they satisfy the assumption of Lemma $4$. Therefore, for any subset $S \subset \{ 1, 2, \ldots , r-1, r+1, \ldots , n-2, n\}$, $\prod_{j \in S} a_{j}$ cannot a perfect square. 
    
    Hence, if there exists a set $S \subset \{1, 2, \ldots , r-1, r+1, \ldots , n\}$ such that $\prod_{j \in S} a_{j}$ is a perfect square, then $(n-1) \in S$. If $n \in S$, then by Lemma $1$, we must have that $a_{n} = \text{rad }\big(\prod_{j \in S, j\neq n} a_{j}\big)$, which is a contradiction to the way $a_{n}$ was chosen in the step $6$ of Section $3$.
    
    Therefore $n-1 = $ max $S$ and by Lemma $1$ we have that 
    
    \begin{equation*}
    a_{n-1} = \text{ rad } \big(\prod_{j \in S, j \neq (n-1)} a_{j}\big).
    \end{equation*}
    
    This, along with Step $5$ of Section $3$ implies 
    \begin{equation*}
        \text{ rad } \big(\prod_{j \in S, j \neq (n-1)} a_{j}\big) = \begin{cases}
        
        \text{rad } \big(\prod_{i=1}^{n-2} a_{j}\big) & \text{ ; if $n$ is even}\\
        
        \text{rad } \big(\prod_{i=1}^{n-3} a_{j}\big) & \text{ ; if $n$ is odd.}
        
        \end{cases}
      \end{equation*}
      
      Hence we obtain that
      
      \begin{equation*}
          \begin{cases}
             \prod_{j \in S, j \neq (n-1)} a_{j} \times \prod_{i=1}^{n-2} a_{j} \text{ is a perfect square} & \text{ ; if $n$ is even}\\
             
             \prod_{j \in S, j \neq (n-1)} a_{j} \times \prod_{i=1}^{n-3} a_{j} \text{ is a perfect square} & \text{ ; if $n$ is odd.}
        
          \end{cases}
      \end{equation*}
      
      Since $r \not\in S$ but $r \in \{1, 2, \ldots, n-3\}$, the two factors above are not equal, regardless of whether $n$ is odd or even. Therefore, we could disregard all $a_{j}$ appearing in both the multiplicands above. 
      
      After disregarding the common $a_{j}$ appearing in both the multiplicands, we are still left with a product of distinct $a_{j}$ that is a perfect square. Let $r_{0}$ be the largest such that $a_{r_{0}}$ remaining in this product. Then using Lemma $1$ we have that
      \begin{equation*}
      a_{r_{0}} = \text{ rad }  \big(a_{j_{1}} \times \cdots \times a_{j_{l}}\big),
      \end{equation*}
    where $j_{1}, \ldots , j_{l} \in \{1, 2, \ldots , r_{0} - 1 \}$. Since $a_{r_{0}}$ was chosen from product of $a_{j}$ with $j \leq (n-2)$, we have that $r_{0} \leq (n-2)$. This is a contradiction to how $a_{1}, a_{2}, \ldots , a_{n-2}$ were chosen the steps $1, 2$, $3$ and $4$ of Section $3$. Specifically, for $i \leq n-2$, we had chosen $a_{i}$ to be greater than square-free part of any sub-product of $a_{1}, \ldots , a_{i-1}$.
    
    Therefore if we remove $(x^{2} - a_{r})$ for any $1 \leq r \leq (n-3)$ then for the resulting polynomial $g_{n}(x):= \frac{f_{n}(x)}{(x^{2} - a_{r})}$, the condition $1(a)$ of Proposition $1$ will not be satisfied. Hence, $g_{n}(x)$ is not intersective.

    \item\underline{\textbf{Removing $(x^{2} - a_{n-2})$ from $f_{n}(x)$}}
    
    If $n$ is even, then removing $(x^{2} - a_{n-2})$ from $f_{n}(x)$ will again result in a polynomial $g_{n}(x):= \frac{f_{n}(x)}{(x^{2} - a_{n-2})}$ not satisfying the condition $1(a)$ of Proposition $1$. The proof of this is exactly analogous to the previous case of removing $(x^{2} - a_{r})$ for $1 \leq r \leq (n-3)$. Therefore, we assume that $n$ is odd.
    
    \begin{itemize}
        \item Since $\big(\frac{a_{1}}{p_{1}}\big) = 0 \neq +1$, $(x^{2} - a_{1}) \equiv 0 \hspace{1mm} (\text{mod}\hspace{1mm} p_{1}^{2})$ is not solvable. This follows from Lemma $2$. 
        
        \item Since for any $2 \leq i \leq n-3$ or $i = n$ $\big(\frac{a_{i}}{p_{1}}\big) = -1 \neq +1$ from step $2, 3$ and $6$ of Section $3$, $(x^{2} - a_{i}) \equiv 0 \hspace{1mm} (\text{mod}\hspace{1mm} p_{1}^{2})$ is not solvable . This again follows from Lemma $2$. 
        
        \item $(x^{2} - a_{n-1}) \equiv 0 \hspace{1mm} (\text{mod}\hspace{1mm} p_{1}^{2})$ is not solvable because $\big(\frac{a_{n-1}}{p_{1}}\big) = 0 \neq +1$. This follows  from Lemma $2$ again. 
    \end{itemize}
    
    So for the resulting polynomial \begin{gather*} 
    g_{n}(x) = (x^{2} - a_{1}) \cdots (x^{2} - a_{n-3}) (x^{2} - a_{n-1}) (x^{2} - a_{n})
    \end{gather*}
    the congruence 
    \begin{gather*}
    g_{n}(x) \equiv 0 \hspace{1mm} (\text{mod}\hspace{1mm} p_{1}^{2(n-1)})
    \end{gather*}
    is not solvable. This last assertion follows from Lemma $3$ for $k = 2$, $m = (n-1)$ and $p = p_{1}$. Hence, $g_{n}(x)$ is not intersective.\\
    
    \item\underline{\textbf{Removing $(x^{2} - a_{n-1})$ from $f_{n}(x)$}}
    
    In this case, we note that $a_{1}, \ldots , a_{n-2}, a_{n}$ satisfies the hypothesis of Lemma $4$ and hence for every subset $S \subset \{ 1, 2, \ldots , n-2, n \}$, the product $\prod_{j \in S} a_{j}$ is not a perfect square. 
    
    Therefore the polynomial $g_{n}(x):= \frac{f_{n}(x)}{(x^{2} - a_{n-1})}$ does not satisfy the condition $1(a)$ of Proposition $1$ and hence is not intersective.\\
    
    \item\underline{\textbf{Removing $(x^{2} - a_{n})$ from $f_{n}(x)$}}
    
    If $n$ is even, then note the following. 
    \begin{itemize}
        \item Since $\big(\frac{a_{1}}{p_{1}}\big) = 0 \neq +1$, $(x^{2} - a_{1}) \equiv 0 \hspace{1mm} (\text{mod}\hspace{1mm} p_{1}^{2})$ is not solvable. This follows from Lemma $2$. 
        
        \item Since for any $2 \leq i \leq n-2$ $\big(\frac{a_{1}}{p_{1}}\big) = -1 \neq +1$ from step $3$ and $4$ of Section $3$, $(x^{2} - a_{i}) \equiv 0 \hspace{1mm} (\text{mod}\hspace{1mm} p_{1}^{2})$ is not solvable . This follows again from Lemma $2$.  
        
        \item $(x^{2} - a_{n-1}) \equiv 0 \hspace{1mm} (\text{mod}\hspace{1mm} p_{1}^{2})$ is not solvable because $\big(\frac{a_{n-1}}{p_{1}}\big) = 0 \neq +1$. This follows again from Lemma $2$. 
    \end{itemize}
    
    Therefore, for the resulting polynomial $g_{n}(x) = (x^{2} - a_{1}) \cdots  (x^{2} - a_{n-2}) (x^{2} - a_{n-1})$ we have that the congruence
    \begin{gather*}
       g_{n}(x) \equiv 0 \hspace{1mm} (\text{mod}\hspace{1mm} p_{1}^{2(n-1)}) 
    \end{gather*}
    is not solvable  This last assertion follows from Lemma $3$ for $k = 2$, $m = (n-1)$ and $p = p_{1}$; hence, $g_{n}(x)$ is not intersective.
    
    If $n$ is odd, then we have the following implications. 
    \begin{itemize}
        \item Since $\big(\frac{a_{1}}{p_{2}}\big) = 0 \neq +1$, $(x^{2} - a_{1}) \equiv 0 \hspace{1mm} (\text{mod}\hspace{1mm} p_{2}^{2})$ is not solvable. This follows from Lemma $2$. 
        
        \item Since for any $2 \leq i \leq n-2$ $\big(\frac{a_{1}}{p_{2}}\big) = -1 \neq +1$ from step $3$ and $4$ of Section $3$, $(x^{2} - a_{i}) \equiv 0 \hspace{1mm} (\text{mod}\hspace{1mm} p_{2}^{2})$ is not solvable. This follows again from Lemma $2$.  
        
        \item $(x^{2} - a_{n-1}) \equiv 0 \hspace{1mm} (\text{mod}\hspace{1mm} p_{2}^{2})$ is not solvable because $\big(\frac{a_{n-1}}{p_{2}}\big) = 0 \neq +1$. This follows again from Lemma $2$. 
    \end{itemize}
    
    Therefore, when $n$ is odd, for  the resulting polynomial
    \begin{gather*}
        g_{n}(x) = (x^{2} - a_{1}) \cdots  (x^{2} - a_{n-2}) (x^{2} - a_{n-1})
    \end{gather*}
    we have that the congruence 
    \begin{gather*}
        g_{n}(x) \equiv 0 \hspace{1mm} (\text{mod}\hspace{1mm} p_{2}^{2(n-1)})
    \end{gather*}
    is not solvable. This last assertion again follows from Lemma $3$ for $k = 2$, $m = (n-1)$ and $p = p_{2}$; hence $g_{n}(x)$ is not intersective.

\end{itemize}

\section{Some Examples}

We shall construct an explicit example of minimally intersective $f_{4}(x)$ and then another of a minimally intersective $f_{5}(x)$. 

\subsection{An Example of Minimally Intersective $f_{4}(x)$}

\begin{enumerate}
    \item We pick $p_{1} = $3 and $p_{2} = 5$. And, we pick $c_{1} = 1 \in Q_{3}$, $b_{1} = 2 \not\in Q_{3}$, $c_{2} = 1 \in Q_{5}$ and $b_{2} = 2 \not\in Q_{5}$. We define $a_{1} = p_{1}p_{2} = 15$.
    
    \item Now we pick a square-free integer $ a_{2} > 15$ such that $a_{2} \equiv 2 \hspace{1mm} (\text{mod}\hspace{1mm} 3)$ and $a_{2} \equiv 2 \hspace{1mm} (\text{mod}\hspace{1mm} 5)$. We pick $a_{2} = 17$.
    
    \item Since $n = 4$ is even and $n - 1 = 3$, we take $a_{3} =$ rad $(a_{1} \times a_{2})$ = rad $( 15\times 17) = 255$. 
    
    \item Now we take all the primes $p_{1} = 3, p_{2} = 5, p_{3} = 17$ that divides any one of the $a_{1}, a_{2}, a_{3}$. Then for every $1 \leq j \leq 3$, we take $c_{j} \in Q_{p_{j}}$. Here we take $c_{1} = 1$, $c_{2} = 1$ and $c_{3} = 2$. 
    
    Then we solve for a square-free $a_{4}$ such that  $a_{4} > \text{rad }(\prod_{J} a_{j})$ for every $J\subseteq\{1, 2, 3\}$, $a_{4} \equiv c_{j} \hspace{1mm} (\text{mod}\hspace{1mm} p_{j})$ for every $j = 1, 2, 3$ and $a_{4} \equiv 1 \hspace{1mm} (\text{mod}\hspace{1mm} 8)$.
    
    By Proposition $2$, infinitely many such square-free $a_{4}$ exists. We choose $a_{4} = 2161 $, which is a prime and hence square-free. Therefore,
    \begin{gather*}
    f_{4}(x) = (x^{2} - 15) (x^{2} - 17) (x^{2} - 255) (x^{2} - 2161)
    \end{gather*}
    is minimally intersective. 
    
\end{enumerate}

\subsection{An Example of Minimally Intersective $f_{5}(x)$}

\begin{enumerate}

    \item As in $6.1$, take $p_{1} = 3$, $p_{2} = 5$, $c_{1} = 1 \in Q_{3}$, $c_{2} = 1 \in Q_{5}$, $b_{1} = 2 \not\in Q_{3}$, $b_{2} = 2 \not\in Q_{5}$, $a_{1} = 15 $ and $a_{2} = 17 $. 
    
    \item Pick a square-free integer $ a_{3} > 15 \times 17 = 255 $ such that $a_{2} \equiv 2 \hspace{1mm} (\text{mod}\hspace{1mm} 3)$ and $a_{2} \equiv 2 \hspace{1mm} (\text{mod}\hspace{1mm} 5)$. We choose $a_{3} = 557$, which is a prime and hence square-free.
    
    \item Define $a_{4}$ to be rad $(a_{1} \times a_{2}) = 15 \times 17 = 255 $ and hence $a_{4} = 255$.
    
    \item Pick all the odd primes $p_{1} = 3, p_{2} = 5, p_{3} = 17, p_{4} = 557$ that divide any one of the $a_{1}, a_{2}, a_{3}, a_{4} $. Then take $c_{1} = 1$, $c_{2} = 1$, $c_{3} = 2$ and $c_{4} = 6$ which are in $Q_{p_{j}}$ for $j = 1, 2, 3, 4$ respectively.
    
    Now choose a square-free $ a_{5} $ such that $a_{5} > \prod_{J} a_{j}$ for all $J \subset\{1, 2, 3, 4\}$, $a_{5} \equiv c_{j} \hspace{1mm} (\text{mod}\hspace{1mm} p_{j})$ for $j = 1, 2, 3, 4$ and $a_{5} \equiv 1 \hspace{1mm} (\text{mod}\hspace{1mm} 8)$.
    
    Any such $a_{5}$ has to be of the form $587641 + 142035k$ for some $k \in\mathbb{Z}$. We take $a_{5} = 587641 + 142035 (2) = 871711 $, which is square-free since its prime-factorization is $29 \times 30059$. Therefore, 
    \begin{gather*}
    f_{5}(x) = (x^{2} - 15) (x^{2} - 17) (x^{2} - 557) (x^{2} - 255) (x^{2} - 871711)
    \end{gather*}
    is minimally intersective. 
\end{enumerate}

\end{document}